\documentclass[12pt]{article}
\usepackage{amsthm}
\usepackage{amssymb}
\usepackage{amsmath}

\newtheorem{thm}{Theorem}
\newtheorem{prop}{Proposition}
\newtheorem{defi}{Definition}

\newtheorem{rem}{Remark}

\newtheorem{lem}{Lemma}

\newcommand{\ai}{\operatorname{Ai}}
\newcommand{\ram}{\operatorname{A}}
\newcommand{\res}{\operatorname{Res}}

\begin{document}
\title{On local solutions of the Ramanujan equation and their connection formulae}
\author{Takeshi MORITA\thanks{Graduate School of Information Science and Technology, Osaka University, 
1-1  Machikaneyama-machi, Toyonaka, 560-0043, Japan.} }
\date{}
\maketitle
\begin{abstract}
We show connection formulae of local solutions of the Ramanujan equation between the origin and the infinity. These solutions are given by the Ramanujan function, the $q$-Airy function and the divergent basic hypergeometric series ${}_2\varphi_0(0,0;-;q,x)$. We use two different $q$-Borel-Laplace transformations to obtain our connection formulae. 
\end{abstract}
\section{Introduction}
In this papar, we show two essentially different connection formulae of some basic hypergeometric series between the origin and the infinity. In 1846, E.~Heine \cite{hein} introduced the basic hypergeometric series ${}_2\varphi_1(a,b;c;q,x)$ as follows;
\begin{equation}
{}_2\varphi_1(a,b;c;q,x):=\sum_{n\ge 0}\frac{(a,b;q)_n}{(c;q)_n(q;q)_n}x^n,\quad c\not \in q^{-\mathbb{N}}.
\label{heine}
\end{equation}
Here, $(a;q)_n$ is the $q$-shifted factorial;
\[(a;q)_n:=
\begin{cases}
1, &n=0, \\
(1-a)(1-aq)\dots (1-aq^{n-1}), &n\ge 1,
\end{cases}
\]
moreover, $(a;q)_\infty :=\lim_{n\to \infty}(a;q)_n$ and 
\[(a_1,a_2,\dots ,a_m;q)_\infty:=(a_1;q)_\infty (a_2;q)_\infty \dots (a_m;q)_\infty.\]
The $q$-shifted factorial $(a;q)_n$ is a $q$-analogue of the shifted factorial $(\alpha )_n$;
\[(\alpha )_n:=
\begin{cases}1, &n=0,\\
\alpha (\alpha +1)\dots \{\alpha +(n-1)\}, &n\ge 1.
\end{cases}\]
The basic hypergeometric series \eqref{heine} is a $q$-analogue of the hypergeometric series ${}_2F_1(\alpha ,\beta ;\gamma ,z)$\cite{Gauss13};
\begin{equation}
{}_2F_1(\alpha ,\beta ;\gamma ,z):=\sum_{n\ge 0}\frac{(\alpha )_n(\beta )_n}{(\gamma )_nn!}z^n.\label{gauss}
\end{equation}
This series \eqref{gauss} has the following famous degeneration diagram
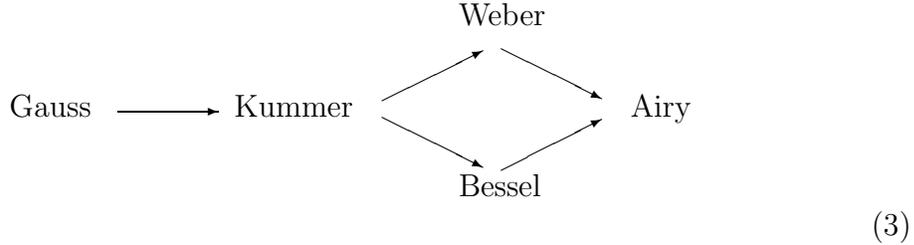
\begin{equation}
	\begin{picture}(370,70)(0,0)
        \put(35,30){{  Gauss } }
        \put(120,30){{  Kummer }}
        \put(205,0){{  Bessel}}
        \put(205,65){{  Weber}}
        \put(270,30){{  Airy}}
        \put(80,32){\vector(1,0){38}}
        \put(180,36){\vector(2,1){38}}
        \put(180,30){\vector(2,-1){38}}
        \put(225,10){\vector(2,1){38}}
        \put(225,56){\vector(2,-1){38}}
     \end{picture}\label{digconti}
\end{equation}

Recently, Y.~Ohyama  \cite{Ohyama} shows that there exists ``the digeneration diagram'' of Heine's series \eqref{heine} as follows:
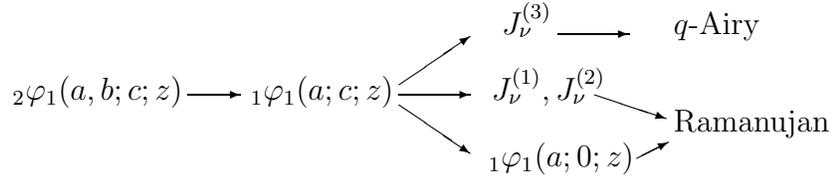
\[
	\begin{picture}(370,70)(0,0)
        \put(20,30){{  $ _2 \varphi_1(a,b;c;z)$ } }
        \put(110,30){{ $ _1 \varphi_1(a;c;z)$   }}
        \put(200,5){{ $ _1 \varphi_1(a;0;z)$   }}
        \put(205,55){{ $J_\nu^{(3)}$  }}
        \put(205,30){{$J_\nu^{(1)},J_{\nu}^{(2)} $}}
        \put(270,55){{  $q$-Airy}}
        \put(270,18){{ Ramanujan}}
        \put(90,32){\vector(1,0){20}}
        \put(170,36){\vector(3,2){27}}
        \put(170,32){\vector(1,0){27}}
        \put(230,55){\vector(1,0){27}}
        \put(170,28){\vector(3,-2){27}}
        \put(260,8){\vector(2,1){12}}
        \put(244,32){\vector(3,-1){27}}
     \end{picture}
\]
We remark that there exist three different $q$-Bessel functions $J_\nu^{(j)},j=1,2,3$\cite{GR} and two $q$-analogues of the Airy function. In this point, this diagram is essentially different from the diagram \eqref{digconti}. 

Ismail has pointed out that the Ramanujan function is one of $q$-analogues of the Airy function \cite{Is}.  The Ramanujan function appears in the third identity on p.$57$ of Ramanujan's ``Lost notebook'' \cite{Ramanujan} as follows (with $x$ replaced by $q$):
\[\ram_q(-a)=\sum_{n\ge 0}\frac{a^nq^{n^2}}{(q;q)_n}=
\prod_{n\ge 1}\left(1+\frac{aq^{2n-1}}{1-q^ny_1-q^{2n}y_2-q^{3n}y_3-\cdots }\right)\]
where
\begin{align*}
y_1&=\frac{1}{(1-q)\psi^2(q)},\\
y_2&=0,\\
y_3&=\frac{q+q^3}{(1-q)(1-q^2)(1-q^3)\psi^2(q)}
-\frac{\sum_{n\ge 0}\frac{(2n+1)q^{2n+1}}{1-q^{2n+1}}}{(1-q)^3\psi^6(q)},\\
y_4&=y_1y_3,\\
\psi (q)&=\sum_{n\ge 0}q^{\frac{n(n+1)}{2}}=\frac{(q^2;q^2)_\infty}{(q;q^2)_\infty}.\\
\end{align*}
To be precise, the Ramanujan function is given by
\[A_q(x):=\sum_{n\ge 0}\frac{q^{n^2}}{(q;q)_n}(-x)^n.\]
This function satisfies the following second order linear $q$-difference equation;
\begin{equation}
qxu(q^2x)-u(qx)+u(x)=0.
\label{ramanujaneq}
\end{equation}
The equation \eqref{ramanujaneq} has another solution which is given by a divergent series
\[\theta_q(x){}_2\varphi_0\left(0,0;-;q,-\frac{x}{q}\right)
=\theta_q(x)\sum_{n\ge 0}\frac{1}{(q;q)_n}\left\{(-1)^nq^{\frac{n(n-1)}{2}}\right\}^{-1}\left(-\frac{x}{q}\right)^n.\]
Here, $\theta_q(\cdot )$ is the theta function of Jacobi (see the section 2).

An asymptotic formula for the Ramanujan function is obtained by M.~E.~H.~Ismail and C.~Zhang as follows\cite{IZ};
\begin{equation}
\ram_q(x)=\frac{(qx,q/x;q^2)_\infty}{(q;q^2)_\infty}{}_1\varphi_1\left(0;q;q^2,\frac{q^2}{x}\right)-\frac{q(q^2x,1/x;q^2)_\infty}{(1-q)(q;q^2)_\infty}{}_1\varphi_1\left(0;q^3;q^2,\frac{q^3}{x}\right).
\label{Is-Zh}
\end{equation}
From the viewpoint of connection problems on $q$-difference equations, we can regard the formula \eqref{Is-Zh} as one of connection formulae of the Ramanujan function.

The other $q$-analogue of the Airy function is known as the $q$-Airy function $\ai_q(\cdot )$. The $q$-Airy function has found in the study of the second $q$-Painlev\'e equation\cite{hama}. The function $\ai_q(\cdot )$ is defined by
\[\ai_q(x):=\sum_{n\ge 0}\frac{1}{(-q;q)_n(q;q)_n}\left\{(-1)^nq^{\frac{n(n-1)}{2}}\right\}(-x)^n\]
and satisfies the following $q$-difference equation
\begin{equation}
u(q^2x)+xu(qx)-u(x)=0.\label{qairy}
\end{equation}
The other solution of the equation \eqref{qairy} around the origin is given by 
\[u(x)=\frac{\theta_q(q^2x)}{\theta_q(-q^2x)}\ai_q(-x).\]
Ismail also has pointed out the Ramanujan function and the $q$-Airy function are different. But the relation between them has not known. In the section 3, we give the connection formula between these functions with using the $q$-Borel-Laplace transformations of the second kind.

\noindent
\textbf{Theorem }\textit{For any }$x\in\mathbb{C}^*$, we have
\[\ram_{q^2}\left(-\frac{q^3}{x^2}\right)
=
\frac{1}{(q,-1;q)_\infty}\left\{\theta\left(\frac{x}{q}\right)\ai_q(-x)
+\theta \left(-\frac{x}{q}\right)\ai_q (x)\right\}.\]

Connection problems on linear $q$-difference equations between the origin and the infinity are studied by G.~D.~Birkhoff \cite{Birkhoff}. The first example of the connection formula was found by G.~N.~Watson \cite{W} in 1912. This formula is known as ``Watson's formula for ${}_2\varphi_1(a,b;c;q,x)$'' as follows \cite{GR};
 \begin{align}\label{wato}
{}_2 \varphi_1\left(a,b;c;q;x \right)= 
\frac{(b,c/a;q)_\infty (a x,q/ a x;q)_\infty }{(c, b/a;q)_\infty (  x,q/   x;q)_\infty } 
{}_2 \varphi_1\left(a,aq/c;aq/b;q;cq/abx \right) \nonumber \\+ 
\frac{(a,c/b;q)_\infty (b x, q/ b x;q)_\infty }{(c, a/b;q)_\infty (  x,q/   x;q)_\infty } 
{}_2 \varphi_1\left(b, bq/c; bq/a; q; cq/abx \right).
\end{align}
But other connection formulae had not found for a long time. Recently, C.~Zhang gives connection formulae for some confluent type basic hypergeometric series \cite{Z0,Z1,Z2}. In \cite{Z1}, Zhang gives a connection formula of Jackson's first and second $q$-Bessel function $J_\nu^{(j)}(x;q), (j=1,2)$;
\[J_\nu^{(1)}(x;q):=\frac{(q^{\nu +1};q)_\infty}{(q;q)_\infty}\left(\frac{x}{2}\right)^\nu\sum_{n\ge 0}\frac{1}{(q^{\nu +1};q)_n}\left(-\frac{x^2}{4}\right)^n\]and
\[J_\nu^{(2)}(x;q):=\frac{(q^{\nu +1};q)_\infty}{(q;q)_\infty}\left(\frac{x}{2}\right)^\nu\sum_{n\ge 0}\frac{q^{n^2}}{(q^{\nu +1};q)_n}\left(-\frac{q^\nu x^2}{4}\right)^n\]
with using the $q$-Borel-Laplace transformations of the second kind $\mathcal{B}_q^-$ and $\mathcal{L}_q^-$. These transformations are defined for a formal power series $f(x)=\sum_{n\ge 0}a_nx^n$, $a_0=1$ as follow;
\begin{enumerate}
\item The $q$-Borel transformation of the second kind is 
\[(\mathcal{B}_q^-f)(\xi ):=\sum_{n\ge 0}a_nq^{-\frac{n(n-1)}{2}}\xi^n\left(=:g(\xi )\right).\]
\item The $q$-Laplace transformation of the second kind is
\[\left(\mathcal{L}_q^-g\right)(x):=\frac{1}{2\pi i}\int_{|\xi |=r}g(\xi )\theta_q\left(\frac{x}{\xi }\right)\frac{d\xi}{\xi},\]
where $r>0$ is enough small number. 
\end{enumerate}

In \cite{M0} and \cite{M1}, we obtained connection formulae of the Hahn-Exton $q$-Bessel function
\[
J_\nu^{(3)}(x;q):=\frac{(q^{\nu +1};q)_\infty}{(q;q)_\infty}x^\nu\sum_{n\ge 0}\frac{q^{\frac{n(n+1)}{2}}}{(q^{\nu +1};q)_n}\left(-x^2\right)^n
\]
and the $q$-confluent type basic hypergeometric function 
\[{}_1\varphi_1(a;b;q,x):=\sum_{n\ge 0}\frac{(a;q)_n}{(b;q)_n(q;q)_n}(-1)^nq^{\frac{n(n-1)}{2}}x^n\]
by these transformations. In section 3, we use these transformations to obtain connection formula between the Ramanujan function and the $q$-Airy function. 

On the other hand, the $q$-Borel-Laplace transformations of the first kind are defined for a formal power series as follow;
\begin{enumerate}
\item The $q$-Borel transformation of the first kind is
\[\left(\mathcal{B}_q^+f\right)(\xi ):=\sum_{n\ge 0}a_nq^{\frac{n(n-1)}{2}}\xi^n\left(=:\varphi (\xi )\right).\]
\item The $q$-Laplace transformation of the first kind is
\[\left(\mathcal{L}_q^+\varphi\right)(x):=
\frac{1}{1-q}\int_0^{\lambda\infty}\frac{\varphi (\xi )}{\theta_q\left(\frac{\xi}{x}\right)}\frac{d_q\xi}{\xi}=\sum_{n\in\mathbb{Z}}\frac{\varphi (\lambda q^n)}{\theta_q\left(\frac{\lambda q^n}{x}\right)},\]
here, this transformation is given by Jackson's $q$-integral \cite{GR}.
\end{enumerate}
These two different types of $q$-Borel-Laplace transformations are introduced by J.~Sauloy \cite{sauloy} and stusied by C.~Zhang.
We remark that each $q$-Borel transformation is formal inverse of each $q$-Laplace transformation, i.e.,
\[\mathcal{L}_q^{\pm}\circ\mathcal{B}_q^{\pm}f=f.\]
 The application of the $q$-Borel-Laplace transformations of the first kind is found in \cite{Z0,Z2}. Zhang gives the connection formula of the divergent basic hypergeometric series ${}_2\varphi_0(a,b;-;q,x)$ as follows;

\noindent
\textbf{Theorem} (Zhang, \cite{Z0}) 
\textit{For any $x\in\mathbb{C}^*$, we have} 
\begin{align*}
{}_2f_0&(a,b;\lambda ,q,x)\\
=&\frac{(b;q)_\infty}{(b/a;q)_\infty}\frac{\theta_q(a\lambda )}{\theta_q(\lambda )}\frac{\theta_q(qax/\lambda )}{\theta_q(\lambda /x)}{}_2\varphi_1\left(a,0;\frac{aq}{b};q,\frac{q}{abx}\right)\\
+&\frac{(a;q)_\infty}{(a/b;q)_\infty}\frac{\theta_q(b\lambda )}{\theta_q(\lambda )}\frac{\theta_q(qbx/\lambda )}{\theta_q(\lambda /x)}{}_2\varphi_1\left(b,0;\frac{bq}{a};q,\frac{q}{abx}\right)
\end{align*}
\textit{where $\lambda\in\mathbb{C}^*\setminus \{-q^n;n\in\mathbb{Z}\}$.}

Here, ${}_2f_0(a,b;\lambda ,q,x)$ in the left-hand side is the $q$-Borel-Laplace transform of the function ${}_2\varphi_0(a,b;-;q,x)$. But other application of this method (of the first kind ) has not known. In the section 4, we show the connection formula of the divergent series
\[{}_2\varphi_0(a,b;-;q,x)=\sum_{n\ge 0}\frac{1}{(q;q)_n}\left\{(-1)^nq^{\frac{n(n-1)}{2}}\right\}^{-1}x^n.\]
This formula is given by the following theorem;

\noindent
\textbf{Theorem}\textit{ For any $x\in\mathbb{C}^*\setminus [-\lambda ;q]$,} 
\begin{align*}\theta_q(x){}_2f_0\left(0,0;-;q,-\frac{x}{q}\right)&=(q;q)_\infty\frac{\theta_q(x) \theta_{q^2}\left(-\frac{\lambda^2}{qx}\right)}{\theta_q\left(-\frac{\lambda}{q}\right)\theta_q\left(\frac{\lambda}{x}\right)}
{}_1\varphi_1\left(0;q;q^2,\frac{q^2}{x}\right)\\
&+\frac{(q;q)_\infty}{1-q}\frac{\theta_q(x) \theta_{q^2}\left(-\frac{\lambda^2}{x}\right)}{\theta_q\left(-\frac{\lambda}{q}\right)\theta_q\left(\frac{\lambda}{x}\right)}\frac{\lambda}{x}
{}_1\varphi_1\left(0;q^3;q^2,\frac{q^3}{x}\right).
\end{align*} 

\section{Basic notations}
In this section, we review our notations. We assume that $q\in\mathbb{C}^*$ satisfies $0<|q|<1$. The $q$-shifted operator $\sigma_q$ is given by $\sigma_qf(x)=f(qx)$. For any fixed $\lambda\in\mathbb{C}^*\setminus q^{\mathbb{Z}}$, the set $[\lambda ;q]$-spiral is $[\lambda ;q]:=\lambda q^{\mathbb{Z}}=\{\lambda q^k;k\in\mathbb{Z}\}$. The (generalized) basic hypergeometric series ${}_r\varphi_s(a_1,\dots ,a_r;b_1,\dots ,b_s;q,x)$ is 
\begin{align*}
{}_r\varphi_s(a_1,\dots ,a_r&;b_1,\dots ,b_s;q,x)\\
&:=\sum_{n\ge 0}\frac{(a_1,\dots ,a_r;q)_n}{(b_1,\dots ,b_s;q)_n(q;q)_n}\left\{(-1)^nq^{\frac{n(n-1)}{2}}\right\}^{1+s-r}x^n.
\end{align*}
This series has radius of convergence $\infty , 1$ or $0$ according to whether $r-s<1, r-s=1$ or $r-s>1$ (see \cite{GR} for further details). In connection problems, the theta function of Jacobi is important. This function is defined by
\[\theta_q(x):=\sum_{n\in\mathbb{Z}}q^{\frac{n(n-1)}{2}}x^n,\qquad x\in\mathbb{C}^*.\]
We denote $\theta_q(\cdot )$ or more shortly $\theta (\cdot )$. The theta function has the following properties;
\begin{enumerate}
\item Jacobi's triple product identity is
\[\theta_q(x)=\left(q,-x,-\frac{q}{x};q\right)_\infty .\]
\item The $q$-difference equation which the theta function satisfies;
\[\theta_q(q^kx)=q^{-\frac{n(n-1)}{2}}x^{-k}\theta_q(x),\quad \forall k\in\mathbb{Z}.\]
\item The inversion formula;
\[\theta_q\left(\frac{1}{x}\right)=\frac{1}{x}\theta_q(x).\]
\end{enumerate}
We remark that the function $\theta (-\lambda x)/\theta (\lambda x)$, $\lambda\in\mathbb{C}^*$ satisfies a $q$-difference equation
\[u(qx)=-u(x)\]
which is also satisfied by the function $u(x)=e^{\pi i\left(\frac{\log x}{\log q}\right)}$. 


\section{Two types of the $q$-analogue of the Airy function and the connection formula}
There are two different $q$-analogue of the Airy function. One is called the Ramanujan function which appears in \cite{Ramanujan}. Ismail \cite{Is} pointed out that the Ramanujan function can be considered as a $q$-analogue of the Airy function. The other one is called the $q$-Airy function which is obtained by K.~ Kajiwara, T. Masuda, M. Noumi, Y. Ohta and Y. Yamada \cite{KMNOY} . In this section, we see the properties of these functions. We explain the reason why they are called $q$-analogue of the Airy function and we show $q$-difference equations which they satisfy.

\subsection{The Ramanujan function $\ram_q(x)$}
The Ramanujan function appears in Ramanujan's ``Lost notebook''~\cite{Ramanujan}.
Ismail has pointed out that the Ramanujan function can be considered as a $q$-analogue of the Airy function.
The Ramanujan function is defined by following convergent series;
\[\ram_q(x):=\sum_{n\ge 0}\frac{q^{n^2}}{(q;q)_n}(-x)^n
={_0\varphi_1}(-;0;q,-qx).\]

In the theory of ordinary differencial equations, the term Plancherel-Rotach asymptotics refers to asymptotics around the largest and smallest zeros. With $x=\sqrt{2n+1}-2^{\frac{1}{2}}3^{\frac{1}{3}}n^{\frac{1}{6}}t$ and for $t\in\mathbb{C}$, the Plancherel-Rotach asymptotic formula for Hermite polynomials $H_n(x)$ is
\begin{equation}
\lim_{n\to +\infty}\frac{e^{-\frac{x^2}{2}}}{3^{\frac{1}{3}}\pi^{-\frac{3}{4}}2^{\frac{n}{2}+\frac{1}{4}}\sqrt{n!}}H_n(x)=\ai (t). \label{pr}
\end{equation}
In \cite{Is}, Ismail shows the $q$-analogue of (\ref{pr});
\begin{prop}One can get 
\[\lim_{n\to\infty}\frac{q^{n^2}}{t^n}h_n(\sinh\xi_n|q)=\ram_q\left(\frac{1}{t^2}\right)\]
where $e^{\xi_n}=tq^{-\frac{n}{2}}$.
\end{prop}
Here, $h_n(\cdot |q)$ is the $q$-Hermite polynomial. 
In this sense, we can deal with the Ramanujan function $\ram_q(x)$ as a $q$-analogue
of the Airy function. The Ramanujan function satisfies the following $q$-diference equation;
\begin{equation}
\left(qx\sigma_q^2-\sigma_q+1\right)u(x)=0.
\label{ram}
\end{equation}

\begin{rem}We remark that another solution of the equation (\ref{ram}) is given by
\[u(x)=\theta (x){}_2\varphi_0(0,0;-;q,-x/q).\]
Here, 
\[{}_2\varphi_0\left(0,0;-;q,-\frac{x}{q}\right)=\sum_{n\ge 0}\frac{1}{(q;q)_n}\left\{(-1)^nq^{\frac{n(n-1)}{2}}\right\}^{-1}\left(-\frac{x}{q}\right)^n\]
is a divergent series.
\end{rem}

\subsection{The $q$-Airy function $\ai_q(x)$}
The $q$-Airy function is found by K. Kajiwara, T. Masuda, M. Noumi, Y. Ohta and Y. Yamada \cite{KMNOY}, in their study of the $q$-Painlev\'e equations. This function is the special solution of the second $q$-Painlev\'e equations and given by the following series
\[\ai_q(x):=\sum_{n\ge 0}\frac{1}{(-q,q;q)_n}\left\{(-1)^nq^\frac{n(n-1)}{2}\right\}(-x)^n={}_1\varphi_1(0;-q;q,-x).\]

T. Hamamoto, K. Kajiwara, N. S. Witte \cite{hama} proved following asymptotic expansions;

\begin{prop}
With $q=e^{-\frac{\delta^3}{2}}$, $x=-2ie^{-\frac{s}{2}\delta^2}$ as $\delta\to 0$,
\[{_1\varphi_1}(0;-q;q,-qx)=2\pi^{\frac{1}{2}}\delta^{-\frac{1}{2}}
e^{-\left(\frac{\pi i}{\delta^3}\right)\ln 2+\left(\frac{\pi i}{2\delta }\right)s+\frac{\pi i}{12}}\left[\ai\left(se^{\frac{\pi i}{3}}\right)+O(\delta^2)\right],\]
\[{_1\varphi_1}(0;-q;q,qx)=2\pi^{\frac{1}{2}}\delta^{-\frac{1}{2}}
e^{-\left(\frac{\pi i}{\delta^3}\right)\ln 2-\left(\frac{\pi i}{2\delta }\right)s-\frac{\pi i}{12}}\left[\ai\left(se^{-\frac{\pi i}{3}}\right)+O(\delta^2)\right]\]
for $s$ in any compact domain of $\mathbb{C}$.  
\end{prop}

Here, $\ai (\cdot )$ is the Airy function. From this proposition, we can regard the $q$-Airy function as a $q$-analogue of the Airy function.

We can easily check out that the $q$-Airy function satisfies the second order linear $q$-difference equation
\begin{equation}
\left(\sigma_q^2+x\sigma_q-1\right)u(x)=0.
\label{qai}
\end{equation}
Another solution of the equation (\ref{qai}) is given by
\[u(x)=
e^{\pi i\left(\frac{\log x}{\log q}\right)}{_1\varphi_1}(0;-q;q,x)
=e^{\pi i\left(\frac{\log x}{\log q}\right)}\ai_q(-x).\]


\subsection{Covering transformations}
We define a covering transformation of a second order linear $q$-difference equation.
\begin{defi}For a $q$-difference equation
\begin{equation}\label{sh}
a(x)u(q^2x)+b(x)u(qx)+c(x)u(x)=0,
\end{equation}
we define the covering transformation as follows
\[t^2:=x,\quad v(t):=u(t^2),\quad p:=\sqrt{q}.\]
\end{defi}
The covering transform of the equation (\ref{sh}) is given by
\[a(t^2)v(p^2t)+b(t^2)v(pt)+c(t^2)v(t)=0.\]
By the covering transformation, the equation 
\[\left(K\cdot x\sigma_q^2-\sigma_q+1\right)u(x)=0\]
is transformed to
\begin{equation}
\left(K\cdot t^2\sigma_p^2-\sigma_p+1\right)v(t)=0,\label{kram}
\end{equation}
where $K$ is a fixed constant in $\mathbb{C}^*$.

\subsection{The $q$-Airy equation around the infinity}
We consider the behavior of the equation (\ref{qai}) around the infinity. We set $x=1/t$ and $z(t)=u(1/t)$. Then $z(t)$ satisfies 
\[\left(-\sigma_q^2+\frac{1}{q^2t}\sigma_q+1\right)z(t)=0.\]
We set $\mathcal{E}(t)=1/\theta (-q^2t)$ and $f(t)=\sum_{n\ge 0}a_nt^n,\quad a_0=1$. We assume that $z(t)$ can be described as 
\[z(t)=\mathcal{E}(t)f(t)=\frac{1}{\theta (-q^2t)}\left(\sum_{n\ge 0}a_nt^n\right).\]

The function $\mathcal{E}(t)$ has the following property;
\begin{lem} For any $t\in\mathbb{C}^*$,
\[\sigma_q\mathcal{E}(t)=-q^2t\mathcal{E}(t),
\qquad \sigma_q^2\mathcal{E}(t)=q^5t^2\mathcal{E}(t).\]
\end{lem}
From this lemma, $f(t)$ satisfies the following equation
\begin{equation}\label{qaeq}
\left(-q^5t^2\sigma_q^2-\sigma_q+1\right)f(t)=0.
\end{equation}
Since (\ref{qaeq}) is the same as (\ref{kram}) for $K=-q^5$, we obtain
\[f(t)={}_0\varphi_1(-;0;q^2,q^5t^2)=\ram_{q^2}(-q^3t^2).\]

We show a connection formula for $f(t)$. In order to obtain a connection formula, we need the $q$-Borel transformation and the $q$-Laplace transformation following Zhang \cite{Z1}.
\subsection{\label{qbql}The $q$-Borel transformation and the $q$-Laplace transformation}
\begin{defi}For $f(t)=\sum_{n\ge 0}a_nt^n$, the $q$-Borel transformation is defined by
\[g(\tau )=\left(\mathcal{B}_q^-f\right)(\tau ):=\sum_{n\ge 0}a_nq^{-\frac{n(n-1)}{2}}\tau^n,\]
and the $q$-Laplace transformation is given by
\[\left(\mathcal{L}_q^-g\right)(t):=\frac{1}{2\pi i}\int_{|\tau |=r}g(\tau )\theta\left(\frac{t}{\tau}\right)\frac{d\tau}{\tau },\qquad 0<r<\frac{1}{|q^2|}.\] 
\end{defi}
 The $q$-Borel transformation can be considered as a formal inverse of the $q$-Laplace transformation.
\begin{lem}For any entire function $f$,
\[\mathcal{L}_q^-\circ\mathcal{B}_q^-f=f.\]
\end{lem}
\begin{proof}We can prove this lemma calculating residues of the $q$-Laplace transformation around the origin. 
\end{proof}

The $q$-Borel transformation has following operational relation;
\begin{lem}\label{orqb}
For any $l,m\in\mathbb{Z}_{\ge 0}$,
\[\mathcal{B}_q^-(t^m\sigma_q^l)=q^{-\frac{m(m-1)}{2}}\tau^m\sigma_q^{l-m}\mathcal{B}_q^-.\]
\end{lem}
\subsection{The connection formula of the $q$-Airy function}
Applying the $q$-Borel transformation in \ref{qbql} to the equation (\ref{kram}) and using lemma \ref{orqb}, we obtain the first order $q$-difference equation
\[g(q\tau )=(1+q^2\tau )(1-q^2\tau )g(\tau ).\]
Since $g(0)=1$, $g(\tau )$ is given by an infinite product
\[g(\tau )=\frac{1}{(-q^2\tau ;q)_\infty(q^2\tau ;q)_\infty}\]
which has single poles at
\[\left\{\tau ;\tau =\pm q^{-2-k},\quad \forall k\in\mathbb{Z}_{\ge 0}\right\}.\]
By Cauchy's residue theorem, the $q$-Laplace transform of $g(\tau )$ is 
\begin{align*}
f(t)=&\frac{1}{2\pi i}\int_{|\tau |=r}g(\tau )\theta\left(\frac{t}{\tau}\right)\frac{d\tau}{\tau }\\
=&-\sum_{k\ge 0}\res\left\{g(\tau )\theta\left(\frac{t}{\tau}\right)\frac{1}{\tau };\tau =-q^{-2-k}\right\}\\
&-\sum_{k\ge 0}\res\left\{g(\tau )\theta\left(\frac{t}{\tau}\right)\frac{1}{\tau };\tau =q^{-2-k}\right\}
\end{align*}
where $0<r<r_0:=1/|q^2|$. We can culculate the residue from lemma \ref{lems}.
\begin{lem}\label{lems}For any $k\in\mathbb{N}$, $\lambda\in\mathbb{C}^*$, one can get;
\begin{enumerate}
\item $\res\left\{\dfrac{1}{\left(\tau /\lambda ;q\right)_\infty}\dfrac{1}{\tau}:\tau =\lambda q^{-k}\right\}
=\dfrac{(-1)^{k+1}q^{\frac{k(k+1)}{2}}}{(q;q)_k (q;q)_\infty}$,
\item $\dfrac{1}{(\lambda q^{-k};q)_\infty}=\dfrac{(-\lambda )^{-k}q^{\frac{k(k+1)}{2}}}{(\lambda ;q)_\infty \left(q/\lambda ;q\right)_k},\quad \lambda \not \in q^{\mathbb{Z}}$.
\end{enumerate}
\end{lem}
Summing up all of residues, we obtain 
\begin{align*}
f(t)&=\frac{\theta (q^2t)}{(q,-1;q)_\infty} {_1\varphi_1}\left(0,-q;q,\frac{1}{t}\right)
+\frac{\theta (-q^2 t)}{(q,-1;q)_\infty} {_1\varphi_1}\left(0,-q;q,-\frac{1}{t}\right).
\end{align*}
We obtain a connection formula for $z(t)=\mathcal{E}(t)f(t)$. Finally, we acquire the following connection formula between the Ramanujan function and the $q$-Airy function.
\begin{thm}
For any $x\in\mathbb{C}^*$,
\[\ram_{q^2}\left(-\frac{q^3}{x^2}\right)
=\frac{1}{(q,-1;q)_\infty}
\left\{\theta \left(\frac{x}{q}\right)\ai_q(-x)+\theta\left(-\frac{x}{q}\right)\ai_q(x)\right\}.\]
\end{thm}
Here, both $\ram_q(x)$ and $\ai_q(x)$ are defined by convergent series on whole of the complex plain. The connection formula above is valid for any $x\in\mathbb{C}^*$.
\section{Connection formula of the divergent series ${}_2\varphi_0(0,0;-;q,\cdot )$}
In this section, we show a connection formula of the divergent series $ {}_2\varphi_0$. This series appears in the second solution of the Ramanujan equation \eqref{ram}. At first, we review two $q$-exponential functions to consider our connection formula. 
\subsection{Two different $q$-exponential functions}
In this section, we review two different $q$-exponential functions from the viewpoint of the connection problems. One of the $q$-exponential function $e_q(x)$ is given by 
\[e_q(x):={}_1\varphi_0 (0;-;q,x)=\sum_{n\ge 0}\frac{x^n}{(q;q)_n}.\]
The other $q$-exponential function $E_q(x)$ is 
\[E_q(x):={}_0\varphi_0(-;-;q,-x)=\sum_{n\ge 0}\frac{q^{\frac{n(n-1)}{2}}}{(q;q)_n}x^n.\]
The function $e_q(x)$ satisfies the following first order $q$-difference equation
\[\left\{\sigma_q-(1-x)\right\}u(x)=0\]
and $E_q(x)$ satisfies 
\[\left\{(1+x)\sigma_q-1\right\}u(x)=0.\]
The limit $q\to 1-0$ converges the exponential function
\[\lim_{q\to 1-0}e_q\left(x(1-q)\right)=\lim_{q\to 1-0}E_q\left(x(1-q)\right)=e^x.\]
In this sense, these functions considered as $q$-analogues of the exponential function. It is known that there exists the relation between these functions:
\[e_q(x)E_q(-x)=1,\quad e_{q^{-1}}(x)=E_q(-qx).\]
But another relation has not known. We show the connection formula between them and give alternate representation of $e_q(\cdot )$.
\subsection{The connection formula and alternate representation}
At first, we show the following connection formula between $e_q(\cdot )$ and $E_q(\cdot )$. 
\begin{thm}\label{eqEq} For any $x\in\mathbb{C}^*\setminus [1;q]$, 
\[e_q(x)=\frac{(q;q)_\infty}{\theta_q(-x)}E_q\left(-\frac{q}{x}\right)\]
where $|x|<1$.
\end{thm}
\begin{proof}The function $e_q(x)$ and $E_q(x)$ have infinite product as follows:
\[e_q(x)=\frac{1}{(x;q)_\infty},\qquad |x|<1\]
and
\[E_q(x)=(-x;q)_\infty .\]
We remark that $e_q(x)$ can be described as
\[e_q(x)=\frac{1}{\theta_q(-x)}\left(q,\frac{q}{x};q\right)_\infty =\frac{(q;q)_\infty}{\theta_q(-x)}E_q\left(-\frac{q}{x}\right)\]
where $|x|<1$. We obtain the conclusion.
\end{proof}
Therefore, these $q$-exponential functions are related by the connection formula between the origin and the infinity. If we replace $x$ by $x/q$, we obtain the following lemma. This is useful to consider the connection problem in the last section. 
\begin{lem}\label{alt}For any $x\in\mathbb{C}^*\setminus [1;q]$, the function $e_q(x/q)$ has the following alternate representation.
\begin{align*}
e_q\left(\frac{x}{q}\right)=
\frac{(q;q)_\infty}{\theta_q\left(-\frac{x}{q}\right)}{}_0\varphi_1\left(-;q;q^2,\frac{q^5}{x^2}\right)
-\frac{(q;q)_\infty}{\theta_q\left(-\frac{x}{q}\right)}\frac{q^2}{(1-q)x}{}_0\varphi_1\left(-;q^3;q^2,\frac{q^7}{x^2}\right).
\end{align*}
\end{lem}
\begin{proof}From theorem \ref{eqEq}, 
\[{}_1\varphi_0\left(0;-;q,\frac{x}{q}\right)=\frac{(q;q)_\infty}{\theta_q\left(-\frac{x}{q}\right)}E_q\left(-\frac{q^2}{x}\right)=\frac{(q;q)_\infty}{\theta_q\left(-\frac{x}{q}\right)}{}_0\varphi_0\left(-;-;q,\frac{q^2}{x}\right).\]
Here,
\[{}_0\varphi_0\left(-;-;q,\frac{q^2}{x}\right)
=\sum_{k\ge 0}\frac{1}{(q;q)_k}(-1)^kq^{\frac{k(k-1)}{2}}\left(\frac{q^2}{x}\right)^k\]
and we remark that $(a;q)_{2k}=(a,aq;q^2)_k$\cite{GR}. By separating the terms with even and odd $k\ge 0$, we obtain the conclusion.
\end{proof}

\subsection{The connection formula of the series ${}_2\varphi_0(0,0;-;q,\cdot )$}
The aim of this section is to give a proof for the following theorem;
\begin{thm}\label{divergentramanujan}For any $x\in\mathbb{C}^*\setminus [-\lambda ;q],$ 
\begin{align*}\theta_q(x){}_2f_0&\left(0,0;-;q,-\frac{x}{q}\right)=(q;q)_\infty\frac{\theta_q(x) \theta_{q^2}\left(-\frac{\lambda^2}{qx}\right)}{\theta_q\left(-\frac{\lambda}{q}\right)\theta_q\left(\frac{\lambda}{x}\right)}
{}_1\varphi_1\left(0;q;q^2,\frac{q^2}{x}\right)\\
&+\frac{(q;q)_\infty}{1-q}\frac{\theta_q(x) \theta_{q^2}\left(-\frac{\lambda^2}{x}\right)}{\theta_q\left(-\frac{\lambda}{q}\right)\theta_q\left(\frac{\lambda}{x}\right)}\frac{\lambda}{x}
{}_1\varphi_1\left(0;q^3;q^2,\frac{q^3}{x}\right).
\end{align*} 
\end{thm}

We define the $q$-Borel-Laplace transformations of the first kind to obtain the connection formula between the origin and the infinity. 
\begin{defi}For any analytic function f(x), the $q$-Borel transformation of the 
first kind $\mathcal{B}_q^+$ is
\[\left(\mathcal{B}_q^+f\right)(\xi ):=\sum_{n\ge 0}a_nq^{\frac{n(n-1)}{2}}\xi^n=:\varphi (\xi ),\]
the $q$-Laplace transformation of the first kind $\mathcal{L}_q^+$ is
\[\left(\mathcal{L}_q^+\varphi\right)(x):=\sum_{n\in\mathbb{Z}}\frac{\varphi (\lambda q^n)}{\theta_q\left(\frac{\lambda q^n}{x}\right)}.\]
\end{defi}
We remark that the $q$-Borel transformation $\mathcal{B}_q^+$ is formal inverse of the $q$-Laplace transformation $\mathcal{L}_q^+$ as follows;
\begin{lem}For any entire function $f(x)$, we have 
\[\mathcal{L}_q^+\circ\mathcal{B}_q^+f=f.\]
\end{lem}

We give the proof of theorem \ref{divergentramanujan}.
\begin{proof}
We apply the $q$-Borel transformation $\mathcal{B}_q^+$ to the divergent 
series $v(x)={}_2\varphi_0(0,0;-;q,-x/q)$. We obtain 
\[\left(\mathcal{B}_q^+v\right)(\xi )={}_1\varphi_0\left(0;-;q,\frac{\xi}{q}\right)=:\varphi (\xi ).\]
From lemma \ref{alt}, 
\[\varphi (\xi )=\frac{(q;q)_\infty}{\theta_q\left(-\frac{\xi}{q}\right)}
{}_0\varphi_1\left(-;q;q^2,\frac{q^5}{\xi^2}\right)-\frac{(q;q)_\infty}{\theta_q\left(-\frac{\xi}{q}\right)}
\frac{q^2}{(1-q)\xi}{}_0\varphi_1\left(-;q^3;q^2,\frac{q^7}{\xi^2}\right)\]
where $|\xi /q|<1$. 

We apply the $q$-Laplace transformation $\mathcal{L}_q^+$ to $\varphi (\xi )$:
\begin{align*}
&\left(\mathcal{L}_q^+\varphi\right)(x)=\sum_{n\in\mathbb{Z}}\frac{\varphi (\lambda q^n)}{\theta_q\left(\frac{\lambda q^n}{x}\right)}
=\sum_{n\in\mathbb{Z}}\frac{{}_1\varphi_0\left(0;-;q,\frac{\lambda q^n}{q}\right)}{\theta_q\left(\frac{\lambda q^n}{x}\right)}\\
&=\frac{(q;q)_\infty}{\theta_q\left(-\frac{\lambda}{q}\right)\theta_q\left(\frac{\lambda}{x}\right)}\sum_{n-m\in\mathbb{Z}}(q^2)^{\frac{(n-m)(n-m-1)}{2}}\left(-\frac{\lambda^2}{qx}\right)^{n-m}\\
&\qquad\qquad\qquad\qquad \times \sum_{m\ge 0}\frac{(-1)^m(q^2)^{\frac{m(m-1)}{2}}}{(q;q^2;q^2)_m}\left(\frac{q^2}{x}\right)^m\\
&-\frac{(q;q)_\infty}{\theta_q\left(-\frac{\lambda}{q}\right)\theta_q\left(\frac{\lambda}{x}\right)}\frac{q^2}{(1-q)\lambda}\sum_{n-m\in\mathbb{Z}}(q^2)^{\frac{(n-m)(n-m-1)}{2}}\left(-\frac{\lambda^2}{q^2x}\right)^{n-m}\\
&\qquad\qquad\qquad\qquad \times
\sum_{m\ge 0}\frac{(-1)^m(q^2)^{\frac{m(m-1)}{2}}}{(q^3,q^2;q^2)_m}\left(\frac{q^3}{x}\right)^m.
\end{align*}
Therefore, 
\begin{align*}
&{}_2f_0\left(0,0;-;q,-\frac{x}{q}\right)
=\mathcal{L}_q^+\circ\mathcal{B}_q^+{}_2\varphi_0\left(0,0;-;q,-\frac{x}{q}\right)\\
&=(q;q)_\infty\frac{\theta_{q^2}\left(-\frac{\lambda^2}{qx}\right)}{\theta_q\left(-\frac{\lambda}{q}\right)\theta_q\left(\frac{\lambda}{x}\right)}{}_1\varphi_1\left(0;q;q^2,\frac{q^2}{x}\right)
+\frac{(q;q)_\infty}{1-q}\frac{\theta_{q^2}\left(-\frac{\lambda^2}{x}\right)}{\theta_q\left(-\frac{\lambda}{q}\right)\theta_q\left(\frac{\lambda}{x}\right)}
{}_1\varphi_1\left(0;q^3;q^2,\frac{q^3}{x}\right).
\end{align*}
We obtain the conclusion.
\end{proof}

\section*{Acknowledgement}
The author would like to thank Professor Yousuke Ohyama whose comments and suggestions were innumerably valuable throughout the course of this study.

\end{document}